\documentclass[11pt]{amsart}
\usepackage{url, calc}
\usepackage{amsmath,amsxtra,amssymb,latexsym,epsfig,amscd,amsthm,fancybox,epsfig}
\usepackage[mathscr]{eucal}
\usepackage{graphicx}
\usepackage{multicol,xcolor}
\usepackage{epsfig} 
\usepackage{epstopdf}
\usepackage{cases}
\usepackage{subfig}
\usepackage{color}
\usepackage{hyperref}
\usepackage{bigints}
\newtheorem{thm}{Theorem}[section]

\newtheorem {conj}{Conjecture}[section]
\newtheorem{lm}{Lemma}[section]
\newtheorem{rmk}{Remark}[section]

\newtheorem{deff}{Definition}[section]

\theoremstyle{definition}

\theoremstyle{remark}

\numberwithin{equation}{section}

\DeclareMathOperator{\Conv}{Conv}
\newcommand{\eps}{\varepsilon}

\newcommand{\I}{\mathcal{I}}

\newcommand{\M}{\mathcal{M}}

\newcommand{\E}{\mathbb{E}}
\newcommand{\BE}{\mathbf{E}}
\newcommand{\BB}{\mathbf{B}}

\newcommand{\BX}{\mathbf{X}}
\newcommand{\bx}{\mathbf{x}}

\newcommand{\ba}{\mathbf{a}}

\newcommand{\bc}{\mathbf{c}}

\newcommand{\N}{\mathbb{N}}

\newcommand{\PP}{\mathbb{P}}

\newcommand{\R}{\mathbb{R}}

\newcommand{\Lom}{\mathcal{L}}

\newcommand{\wtd}{\widetilde}
\numberwithin{equation}{section}

\newcommand{\bed}{\begin{displaymath}}
\newcommand{\eed}{\end{displaymath}}
\newcommand{\bea}{\bed\begin{array}{rl}}
\newcommand{\eea}{\end{array}\eed}

\newcommand{\barray}{\begin{array}{ll}}
\newcommand{\earray}{\end{array}}

\newcommand{\1}{\boldsymbol{1}}

\newcommand{\bdelta}{\boldsymbol{\delta}}

\def\hat{\widehat}
\def\a.s{\text{\;a.s.\;}}

\title[Stochastic food chains]{Stochastic Lotka-Volterra food chains}
\author[A. Hening]{Alexandru Hening }
\address{Department of Mathematics \\
 Imperial College London\\
 South Kensington Campus \\
 London SW7 2AZ \\
 United Kingdom}
 \email{a.hening@imperial.ac.uk}

\author[D. Nguyen]{Dang H. Nguyen }
\thanks{D. Nguyen was in part supported by
 the National Science Foundation
under grant DMS-1207667.}
\address{Department of Mathematics \\
 Wayne State University\\
 Detroit, MI 48202 \\
 United States}
 \email{dangnh.maths@gmail.com}


\begin{document}
\maketitle

\begin{abstract}
We study the persistence and extinction of species in a simple food chain that is modelled by a Lotka-Volterra system with environmental stochasticity. There exist sharp results for deterministic Lotka-Volterra systems in the literature but few for their stochastic counterparts. The food chain we analyze consists of one prey and $n-1$ predators. The $j$th predator eats the $j-1$th species and is eaten by the $j+1$th predator; this way each species only interacts with at most two other species - the ones that are immediately above or below it in the trophic chain.
We show that one can classify, based on an explicit quantity depending on the interaction coefficients of the system, which species go extinct and which converge to their unique invariant probability measure. Our work can be seen as a natural extension of the deterministic results of Gard and Hallam '79 to a stochastic setting.

As one consequence we show that environmental stochasticity makes species more likely to go extinct. However, if the environmental fluctuations are small, persistence in the deterministic setting is preserved in the stochastic system. Our analysis also shows that the addition of a new apex predator makes, as expected, the different species more prone to extinction.

Another novelty of our analysis is the fact that we can describe the behavior the system when the noise is degenerate. This is relevant because of the possibility of strong correlations between the effects of the environment on the different species.
\end{abstract}
\tableofcontents

\section{Introduction}
A fundamental problem in ecology is to determine which species go extinct and which persist in a given ecosystem. The complex interactions of the species of the community are sometimes depicted by using food webs. The food web shows the different interaction paths that connect various animals and plants. In general, food webs can be seen as unions of food chains - diagrams that follow one path of interactions. Since species interact in complex ways, food chains and food webs are simplified models of the real world, or caricatures of nature. Nevertheless, one can explain some key properties of an ecosystem by studying these simplified models. To quote \cite{P82}

\textit{``Like caricatures, though their representation of nature is distorted, there is enough truth to permit a study of some of the features they represent.''}

 One well-known model for the interaction of a predator and its prey, the Lotka-Volterra model, has been developed by \cite{L25} and \cite{V28}. Even though the model of \cite{L25} and \cite{V28} is for one predator and one prey one can easily extend it to food chains of any length.
 Many of the food chain models studied in the literature are deterministic and of Lotka-Volterra type. \cite{GH79} give criteria for persistence and extinction for Lotka-Volterra food chains while the global stability of nonnegative equilibrium points is studied by \cite{S79, H79}. More general deterministic food-chains have been analyzed by \cite{G80, FS85}.

In this paper we analyse simple food chains of arbitrary length. We assume that there is only one species at each trophic level and that each species eats only the one on the adjacent lower trophic level. Furthermore, the ecosystem is supposed to have no immigration or emigration.

Our starting point will be the deterministic Lotka-Volterra system (which has been studied by \cite{GH79})
\begin{equation}\label{e:det}
\begin{split}
dx_1(t) &= x_1(t)(a_{10} -a_{11}x_1(t) - a_{12}x_2(t))\,dt\\
dx_2(t) &= x_2(t)(-a_{20} + a_{21}x_1(t) - a_{23}x_3(t))\,dt\\
&\mathrel{\makebox[\widthof{=}]{\vdots}} \\
dx_{n-1}(t) &= x_{n-1}(t)(-a_{n-1,0}+a_{n-1,n-2}x_{n-2}(t)  - a_{n-1,n}x_n)\,dt\\
dx_n(t) &= x_n(t)(-a_{n0} + a_{n,n-1}x_{n-1}(t))\,dt.
\end{split}
\end{equation}

The quantities $(x_1(t),\dots,x_n(t))$ represent the densities of the $n$ species at time $t\geq 0$.
In this model $x_1$ describes a prey species, which is at the bottom of the food chain. The next $n-1$ species are predators. Species $1$ has a per-capita growth rate $a_{10}>0$ and its members compete for resources according to the intra-competition rate $a_{11}>0$. Predator species $j$ has a death rate $-a_{j0}<0$, preys upon species $j-1$ at rate $a_{j,j-1}>0$ and is preyed upon by predator $j+1$ at rate $a_{j,j+1}>0$. The last species, $x_n$, is considered to be the apex predator of the food chain.

\begin{rmk}
We note that the above system ignores the intraspecies competition between predators of the same species. The analysis for the setting with intraspecies competituon is different and will be the subject of the future paper \cite{HN17b}.
\end{rmk}

In the deterministic setting one says that the system \eqref{e:det} is \textit{persistent} if each solution of $\bx(t) = (x_1(t),\dots,x_n(t))$ with $\bx(0)\in \R_+^{n,\circ}:=\{(y_1,\dots,y_n): y_i>0, i=1,\dots,n\}$ satisfies
\[
\limsup_{t\to\infty} x_i(t) >0, i=1,\dots,n.
\]
Species $i$ goes \textit{extinct} if
\[
\lim_{t\to \infty} x_i(t) = 0.
\]

\cite{GH79} are able to prove using ODE and dynamical systems techniques that the persistence or extinction of \eqref{e:det} can be determined using a single parameter which depends on the interaction coefficients $(a_{ij})$. Define

\begin{equation}\label{e:kappa}
\begin{split}
\kappa(n)=\kappa &= a_{10} - \frac{a_{11}}{a_{21}}\left[a_{20}+\sum_{j=2}^k\left(\prod_{i=2}^j\frac{a_{2i-2,2i-1}}{a_{2i,2i-1}}\right)a_{2j,0}\right] - \sum_{j=1}^l\left(\prod_{i=1}^j \frac{a_{2i-1,2i}}{a_{2i+1,2i}}\right)a_{2j+1,0}
\end{split}
\end{equation}
where
\[
k =
\left\{
	\begin{array}{ll}
		n/2, & \mbox{if } n ~\text{even} \\
		(n-1)/2, & \mbox{if } n ~\text{odd}
	\end{array}
\right.
\]
and
\[
l =
\left\{
	\begin{array}{ll}
		n/2-1, & \mbox{if } n ~\text{even} \\
		(n-1)/2, & \mbox{if } n ~\text{odd}.
	\end{array}
\right.
\]
The following theorem is one of the main results of \cite{GH79}.
\begin{thm}\label{t:det}
The food chain modelled by \eqref{e:det} is persistent when $\kappa(n)>0$; it is not persistent (that is, some species go extinct) if $\kappa(n)<0.$
\end{thm}

In nature, the dynamics of species is inherently stochastic due to the random fluctuations of the environmental factors.
The combined effects of biotic interactions and environmental fluctuations are key when trying to determine species richness. Sometimes biotic effects can result in species going extinct. However, if one adds the effects of a random environment extinction might be reversed into coexistence. In other instances deterministic systems that coexist become extinct once one takes into account environmental fluctuations. A  successful way of studying this interplay is modelling the populations as discrete or continuous time Markov processes and looking at the long-term behavior of these processes (\cite{C00, ERSS13, EHS15,  LES03, SLS09, SBA11, BEM07, BS09, BHS08, CM10, CCA09}).

In order to take into account environmental fluctuatuons and their effect on the persistence or extinction of species one approach is to study systems that have random environmental perturbations. This can be done by studying stochastic differential equations that arise by adding noise to ordinary differential equations. For \textit{compact} state spaces \cite{SBA11} provide results for persistence. These results have been generalized by \cite{HN16} where the authors show how, under some natural assumptions, one can characterize the coexistence and extinction of species living on \textit{non-compact} state spaces. Some of these results hold not only for stochastic differential equations but also for stochastic difference equations (see \cite{SBA11}, piecewise deterministic Markov processes (see \cite{HN16, BL16}) and for general Markov processes (see \cite{B14}).

A natural stochastic analogue of \eqref{e:det} is the system

\begin{equation}\label{e:stoc}
\begin{split}
dX_1(t) &= X_1(t)(a_{10} -a_{11}X_1(t)- a_{12}X_2(t))\,dt + X_1(t)\,dE_1(t)\\
dX_2(t) &= X_2(t)(-a_{20} + a_{21}X_1(t) - a_{23}X_3(t))\,dt+X_2(t)\,dE_2(t)\\
&\mathrel{\makebox[\widthof{=}]{\vdots}} \\
dX_{n-1}(t) &= X_{n-1}(t)(-a_{n-1,0}+a_{n-1,n-2}X_{n-2}(t) - a_{n-1,n}X_n)\,dt+X_{n-1}(t)\,dE_{n-1}(t)\\
dX_n(t) &= X_n(t)(-a_{n0} + a_{n,n-1}X_{n-1}(t)+ X_n(t)\,dE_n(t)
\end{split}
\end{equation}
where $\BE(t)=(E_1(t),\dots, E_n(t))^T=\Gamma^\top\BB(t)$ for an $n\times n$ matrix
$\Gamma$ such that
$\Gamma^\top\Gamma=\Sigma=(\sigma_{ij})_{n\times n}$
and $\BB(t)=(B_1(t),\dots, B_n(t))$ is a vector of independent standard Brownian motions. We denote by $\bdelta^*$ the probability measure putting all of its mass at the origin $(0,\dots,0)$.
\begin{rmk}
There are a few different ways to add stochastic noise to deterministic population dynamics. We assume that the environment mainly affects the growth/death rates of the populations. This way, the growth/death rates in an ODE (ordinary differential equation) model are replaced by their average values to which one adds a random noise fluctuation term. See \cite{T77, B02, G88, HNY16, EHS15, ERSS13, SBA11, HN16, G84} for more details.
\end{rmk}
 Define the stochastic growth rate $\tilde a_{10} := a_{10}-\frac{\sigma_{11}}{2}$ and the stochastic death rates $\tilde a_{j0} := a_{j0} + \frac{\sigma_{jj}}{2}, j\geq 2$. For fixed $j\in \{1,\dots,n\}$ we will see that the system
\begin{equation}\label{e:system}
\begin{split}
-a_{11}x_1 - a_{12}x_2 &= -\tilde a_{10}\\
a_{21}x_1  - a_{23}x_3 &=\tilde a_{20}\\
&\mathrel{\makebox[\widthof{=}]{\vdots}} \\
a_{j-1,j-2}x_{j-2}  -a_{j-1,j}x_{j} &= \tilde a_{j-1,0}\\
a_{j,j-1}x_{j-1} &= \tilde a_{j0}
\end{split}
\end{equation}
is intricately related to the stationary distributions of \eqref{e:stoc}.
It is easy to show that \eqref{e:system} has a unique solution, say $\left(x^{(j)}_1,\dots,x^{(j)}_j\right)$.
Define
\begin{equation}\label{e:inv_j+1}
\I_{j+1} = -\tilde a_{j+1,0} + a_{j+1,j} x^{(j)}_j.
\end{equation}
We will show that when \eqref{e:system} has a strictly positive solution the invasion rate of predator $X_{j+1}$ in the habitat of $(X_1,\dots,X_j)$ is exactly $\I_{j+1}$.
The invasion rate of of predator $X_{j+1}$ is the asymptotic logarithmic growth $\lim_{t\to\infty}\frac{\log X_{j+1}(t)}{t}$ when $X_{j+1}$ is introduced at a low density in $(X_1,\dots,X_j)$.
We also set $\I_1:=\tilde a_{10}$ to be the stochastic growth rate of the prey - this can be seen as the invasion rate of the prey into the habitat, when it is introduced at low densities.

Throughout the paper we define for $j=1,\dots,n$ $$\R_+^{(j)}:=\{\bx=(x_1,\dots,x_n)\in\R^n_+: x_{k}=0\,\text{ for } j<k\leq n\},$$
and
$$\R_+^{(j),\circ}:=\{\bx=(x_1,\dots,x_n)\in\R^n_+: x_k>0\,\text{ for } k\leq j; x_{k}=0\,\text{ for } j<k\leq n\}.$$
\begin{rmk}
We will show in Section \ref{s:proof} that \eqref{e:stoc} has unique strong positive solutions.
Furthermore, we show that if  $\BX:=(X_1,\dots,X_n)$ has an invariant probability measure $\pi$ on $\R_+^{(j),\circ}$ then
\[
\E_{\pi}X_i=\int_{\R^n_+} x_i\pi(d\bx)=x_i^{(j)} \,\text{ for } i\leq j.
\]
That is, the solution of \eqref{e:system} is the vector $(\E_{\pi} X_1,\dots, \E_{\pi} X_j)$ of the expected values of $(X_1,\dots,X_j)$ at stationarity.
\end{rmk}

We define the stochastic analogue $\tilde \kappa $ of $\kappa$ via
\begin{equation}\label{e:tkappa}
\begin{split}
\tilde \kappa(n)=\tilde \kappa &= \tilde a_{10} - \frac{a_{11}}{a_{21}}\left[\tilde a_{20}+\sum_{j=2}^k\left(\prod_{i=2}^j\frac{a_{2i-2,2i-1}}{a_{2i,2i-1}}\right)\tilde a_{2j,0}\right] - \sum_{j=1}^l\left(\prod_{i=1}^j \frac{a_{2i-1,2i}}{a_{2i+1,2i}}\right)\tilde a_{2j+1,0}, k\geq 2.
\end{split}
\end{equation}
where
\[
k =
\left\{
	\begin{array}{ll}
		n/2, & \mbox{if } n ~\text{even} \\
		(n-1)/2, & \mbox{if } n ~\text{odd}
	\end{array}
\right.
\]
and
\[
l =
\left\{
	\begin{array}{ll}
		n/2-1, & \mbox{if } n ~\text{even} \\
		(n-1)/2, & \mbox{if } n ~\text{odd}.
	\end{array}
\right.
\]
For notational simplicity we also define $\tilde\kappa(1):=\tilde a_{10}$.

\begin{rmk}\label{r:kap}
By comparing equations \eqref{e:kappa} and \eqref{e:tkappa} one notes that $\tilde\kappa(n)$ is what one gets if one does the substitutions $a_{i0}\mapsto \tilde a_{i0}, i=1,\dots,n$ in the expression for $\kappa(n)$.
Note that
\[
\begin{split}
\tilde \kappa(n) &= \kappa(n) - \frac{1}{2}\sigma_{11} - \frac{a_{11}}{a_{21}}\left[ \frac{1}{2}\sigma_{22} +\sum_{j=2}^k\left(\prod_{i=2}^j\frac{a_{2i-2,2i-1}}{a_{2i,2i-1}}\right) \frac{\sigma_{2j,2j}}{2}\right] - \sum_{j=1}^l\left(\prod_{i=1}^j \frac{a_{2i-1,2i}}{a_{2i+1,2i}}\right)\frac{\sigma_{2j+1,2j+1}}{2}\\
&<\kappa(n).
\end{split}
\]
Furthermore, if we add one extra predator $X_{n+1}$ to the food chain then one has
\[
\tilde \kappa (n+1)=
\left\{
	\begin{array}{ll}
		\tilde \kappa (n) -\left(\prod_{i=1}^{n/2} \frac{a_{2i-1,2i}}{a_{2i+1,2i}}\right)\tilde a_{n+1,0}, & \mbox{if } n ~\text{even} \\
		 \tilde \kappa (n) -\frac{a_{11}}{a_{21}}\left(\prod_{i=1}^{(n+1)/2} \frac{a_{2i-2,2i-1}}{a_{2i,2i-1}}\right)\tilde a_{n+1,0}, & \mbox{if } n ~\text{odd}
	\end{array}
\right.
\]
and in particular $\tilde \kappa(n+1)<\tilde \kappa (n).$
\end{rmk}
There are different concepts regarding the persistence and extinction of species. We review some of these below.
\begin{deff}
The food chain $\BX$ is \textbf{strongly stochastically persistent} if it has a unique invariant probability measure $\pi^*$ on $\R^{n,\circ}_+$ and
\begin{equation}
\lim\limits_{t\to\infty} \|P_\BX(t, \mathbf{x}, \cdot)-\pi^*(\cdot)\|_{\text{TV}}=0, \;\mathbf{x}\in\R^{n,\circ}_+
\end{equation}
where $\|\cdot,\cdot\|_{\text{TV}}$ is the total variation norm and $P_\BX(t,\mathbf{x},\cdot)$ is the transition probability of $(\BX(t))_{t\geq 0}$.
\end{deff}
\begin{deff}
The species $X_k$ goes \textbf{extinct} if for all $\bx\in\R^{n,\circ}_+$
\[
\PP_\bx\left\{\lim_{t\to\infty}X_k(t)=0\right\}=1.
\]
\end{deff}
\begin{deff}
The species $X_k$ goes \textbf{extinct weakly in mean} if
\[
\lim_{t\to\infty}\dfrac1t\int_0^t\E_\bx\left[X_k(s)\right] \,ds= 0.
\]
\end{deff}
\begin{deff}
The species $X_k$ is \textbf{weakly persistent in mean} if
\[
\lim_{t\to\infty}\dfrac1t\int_0^t\E_\bx\left[X_k(s)\right] \,ds> 0.
\]
\end{deff}
\begin{deff}
The species $(X_1,\dots,X_{j^*})$ are \textbf{time-average persistent in probability} if
for any $\eps>0$, there exists a compact set $K_\eps\subset\R^{(j^*),\circ}_+$ such that
\[
\liminf_{t\to\infty}\dfrac1t\int_0^t\PP_\bx\left\{(X_1(s),\dots, X_k(s))\in K_\eps\right\}\,ds \geq 1-\eps
\]
where $\left(x^{(j^*)}_1,\dots,x^{(j^*)}_{j^*}\right)\in\R_+^{(j^*),\circ} $is the unique solution to \eqref{e:system} with $j=j^*$
\end{deff}
We refer the reader to \cite{S12} for a discussion of various forms of persistence.
Having defined all the necessary concepts we can present the main result of this paper.
\begin{thm}\label{t:stoc}
Assume that $a_{11}>0$ and $\BX(0)=\bx\in\R_+^{n,\circ}$.
\begin{itemize}
\item [(i)] If $\tilde \kappa(n)>0$ the food chain $\BX$ modelled by \eqref{e:stoc}
is time-average persistent in probability.
Moreover,
\[
\lim_{t\to\infty}\dfrac1t\int_0^t\E_\bx\left[X_k(s)\right]\,ds = x^{(n)}_k>0, k=1,\dots,n
\]
where $\left(x^{(n)}_1,\dots,x^{(n)}_{n}\right)\in\R_+^{(n),\circ} $is the unique solution of \eqref{e:system} with $j=n$.

Moreover, if $\Sigma$ is positive definite, then the food chain $\BX$
is strongly stochastically persistent and converges to its unique invariant probability measure $\pi^{(n)}$ on $\R_+^{n,\circ}$.
\item [(ii)] If there exists $j^*<n$ such that $\tilde \kappa(j^*)>0$ and $\tilde \kappa(j^*+1)\leq0$ then the predators $(X_{j^*+1},\dots,X_n)$ go weakly extinct in mean, that is
\[
\lim_{t\to\infty}\dfrac1t\int_0^t\E_\bx\left[X_k(s)\right] \,ds= 0, k>j^*.
\]
At the same time, the species $(X_1,\dots,X_{j^*})$ are time-average persistent in probability
and
\[
\lim_{t\to\infty}\dfrac1t\int_0^t\E_\bx\left[X_k(s)\right]\,ds = x^{(j^*)}_k>0, k\leq j^*
\]
where $\left(x^{(j^*)}_1,\dots,x^{(j^*)}_{j^*}\right)\in\R_+^{(j^*),\circ} $is the unique solution to \eqref{e:system} with $j=j^*$.
\item [(iii)] If $n=2$ we can strengthen the extinction results as follows:
\begin{itemize}
\item If $\I_1\leq 0$ then for any $\BX(0)=\bx\in\R_+^{n,\circ}$ we have that $\PP_\bx$-almost surely the randomized occupation measures $\left(\tilde\Pi_t(\cdot)\right)_{t\geq 0}$ converge weakly to $\bdelta^*$ as $t\to\infty$ and
$$\PP_\bx\left\{ \lim_{t\to\infty}\dfrac{\ln X_1(t)}t=\I_1=\tilde a_{10},\lim_{t\to\infty}\dfrac{\ln X_2(t)}t=-\tilde a_{20}\right\}=1.$$
In particular, $X_2$ goes extinct almost surely exponentially fast. If $\I_1<0$ then we also have that $X_1$ goes extinct almost surely exponentially fast .
\item If $\I_1>0$ and $\I_2< 0$  then for any $\BX(0)=\bx\in\R_+^{n,\circ}$ we have that $\PP_\bx$-almost surely the randomized occupation measures $\left(\tilde\Pi_t(\cdot)\right)_{t\geq 0}$ converge weakly, as $t\to\infty$, to the unique invariant probability measure $\pi^{(1)}$ on $\R^{(1),\circ}_+$ and
$$\PP_\bx\left\{\lim_{t\to\infty}\dfrac{\ln X_2(t)}t=\I_2\right\}=1$$
so that $X_2$ goes extinct almost surely exponentialy fast.
\end{itemize}
\end{itemize}
\end{thm}
\begin{rmk}\label{r:extra_predator}
We note that by Theorem \ref{t:stoc} the food chain persists when $\I_j>0, j=2,\dots,n$ and goes weakly extinct when $\I_{j^*+1}<1$ for some $j^*\leq n-1$.  It is key to note that $\I_j$ is independent of the coefficients $(a_{lm}), l>j$.

As such, if we add one extra predator at the top of the food chain the quantities  $\I_j>0, j=2,\dots,n$ remain unchanged and we get one extra invasion rate $\I_{n+1}$.
In this setting, when we have $n$ predators, the system persists when $\I_j>0, j=2,\dots,n$ and $\I_{n+1}>0$ and goes extinct when  $I_{j^*+1}<1$ for some $j^*\leq n$. This means that the introduction of an apex predator makes extinction \textbf{more likely}.
\end{rmk}

\begin{rmk}
Having $\Sigma$ positive definite guarantees that the system \eqref{e:stoc} is nondegenerate and that the noise is truly $n$ dimensional. Otherwise the noise is degenerate.
\end{rmk}

Theorem \ref{t:stoc} extends previous results on stochastic Lotka-Volterra systems in two or three dimensions (see \cite{LB16, HN16, R03}) to an $n$ dimensional setting. We also generalize the work by \cite{G84} where the author gives sufficient conditions for stochastic boundedness persistence of stochastic Lotka-Volterra type food web models in bounded regions of state space. We note that the main results by \cite{G84} only say something about persistence until the first exit time of the process from a compact rectangular region $R_\gamma\subset \R_+^{n,\circ}$. Once the process exits the region, one cannot say whether the species persist or not. Partial results for the existence of invariant probability measures for stochastic Lotka-Volterra systems have been found in \cite{P79}. However, these conditions are quite restrictive and impose artificial constraints on the interaction coefficients. In contrast, our results for persistence and extinction are sufficient and  necessary. Moreover, based on which conditions are satisfied, we can say exactly which species persist and which go extinct.

\section{Mathematical framework}
We rewrite \eqref{e:stoc} as

\begin{equation}\label{e:system_2}
dX_i(t)=X_i(t) f_i(\BX(t))dt+X_i(t)dE_i(t), ~i=1,\dots,n
\end{equation}
where $\BX:=(X_1(t),\dots,X_n(t))$. This is a stochastic process that takes values in $\R_+^{n,\circ}:=(0,\infty)^n$.

The random normalized occupation measures are defined as $$\wtd \Pi_t(B):=\dfrac1t\int_0^t\1_{\{\BX(s)\in\cdot\}}ds,\,t>0, B\in\mathcal{B}(\R_+^{n,\circ})$$
where $\mathcal{B}(\R_+^{n,\circ})$ is the set of all Borel measurable subsets of $\R_+^{n,\circ}$. Note that $\wtd \Pi_t(B)$ tells us the fraction of time the process $\BX$ spends in the set $B$ during the interval $[0,t]$.

Let $\M$ be the set of ergodic invariant probability measures of $\BX$ supported on the boundary $\partial\R^n_+:=\R_+^n\setminus \R_+^{n,\circ}$. For a subset $\wtd\M\subset \M$, denote by $\Conv(\wtd\M)$ the convex hull of $\wtd\M$,
that is the set of probability measures $\pi$ of the form
$\pi(\cdot)=\sum_{\mu\in\wtd\M}p_\mu\mu(\cdot)$
with $p_\mu>0,\sum_{\mu\in\wtd\M}p_\mu=1$.

Note that each subspace of $\R^n_+$ of the form
$$\Big\{(x_1,\dots,x_n)\in\R^n_+: x_i>0 \text{ for } i\in\{\tilde n_1,\dots,\tilde n_k\}; \text{ and }x_i=0\text{ if } i\notin\{\tilde n_1,\dots,\tilde n_k\} \Big\}$$
for some $\tilde n_1,\dots,\tilde n_k\in\N$
satisfying $0<\tilde n_1<\dots<\tilde n_k\leq n$
is an invariant set for the process $\BX$.

As a result any ergodic measure $\mu\in\M$
must be supported in a subspace of this form. More specifically,
there exist $0<n_1<\dots< n_k\leq n$
(if $k=0$ there are no $n_1,\dots, n_k$)
such that $\mu(\R^{\mu,\circ}_+)=1$ where
$$\R_+^\mu:=\{(x_1,\dots,x_n)\in\R^n_+: x_i=0\text{ if } i\in I_\mu^c\}$$
for
$I_\mu:=\{n_1,\dots, n_k\}$,
$I_\mu^c:=\{1,\dots,n\}\setminus\{n_1,\dots, n_k\}$,
$$\R_+^{\mu,\circ}:=\{(x_1,\dots,x_n)\in\R^n_+: x_i=0\text{ if } i\in I_\mu^c\text{ and }x_i>0\text{ if  }x_i\in I_\mu\},$$ and $\partial\R_+^{\mu}:=\R_+^\mu\setminus\R_+^{\mu,\circ}$.
For the Dirac-measure $\bdelta^*$ concentrated at the origin $0$, we have $\I_{\bdelta^*}=\emptyset$.

\begin{rmk}
Note that $\Conv(\M)$ is exactly the set of invariant probability measures of the process $\BX$ supported on the boundary $\partial \R_+^{n}$.
\end{rmk}

For a probability measure $\mu$ on $\R^n_+$, we define the $i$th Lyapunov exponent (when it exists) via
\begin{equation}\label{Lya.exp}
\begin{aligned}
\lambda_j(\mu):=&\int_{\R^n_+}\left(f_j(\bx)-\dfrac{\sigma_{jj}}2\right)\mu(d\bx)
=&
\begin{cases}
\int_{\R^n_+}\left(\tilde a_{10}-a_{11}x_1 - a_{12}x_2\right)\mu(d\bx) &\text{ if }\, j=1,\\
\int_{\R^n_+}\left(-\tilde a_{n0}+a_{n,n-1}x_{n-1}\right)\mu(d\bx)& \text{ if } j=n,\\
\int_{\R^n_+}\left(-\tilde a_{j,0}+a_{j,j-1}x_{j-1}  -a_{j,j+1}x_{j+1}\right)\mu(d\bx)& \text{ otherwise}.
\end{cases}
\end{aligned}
\end{equation}
\begin{rmk}
To determine the Lyapunov exponents of an ergodic invariant measure $\mu\in\M$,
one can look at the equation for $\ln X_i(t)$. An application of It\^o's Lemma yields that		
$$
\dfrac{\ln X_i(t)}t=\dfrac{\ln X_i(0)}t+\dfrac1t\int_0^t\left[f_i(\BX(s))-\dfrac{\sigma_{ii}}2\right]ds+\dfrac1t\int_0^t dE_i(s).
$$

If $\BX$ is close to the support of an ergodic invariant measure $\mu$ for a long time $t\gg 1$,
then
$$\dfrac1t\int_0^t\left[f_i(\BX(s))-\dfrac{\sigma_{ii}}2\right]ds$$
can be approximated by the average with respect to $\mu$
$$\lambda_i(\mu)=\int_{\partial \R^n_+}\left(f_i(\bx)-\dfrac{\sigma_{ii}}2\right)\mu(d\bx).$$
On the other hand, the term $$\dfrac{\ln X_i(0)}t+\frac{E_i(t)}{ t}$$ is negligible for large $t$ since
\[
\PP_\bx \left\{\lim_{t\to\infty} \left(\dfrac{\ln X_i(0)}t+\frac{E_i(t)}{ t}\right)=0\right\}=1.
\]
This implies that $\lambda(\mu_i), i=1,\dots, n$ are the Lyapunov exponents of $\mu$.
\end{rmk}



\begin{rmk}
Straightforward computations show that for all $\bc\in \R_+^{n,\circ}$ and $\gamma_b>0$
\[
\limsup\limits_{\|x\|\to\infty}\left[\dfrac{\sum_i c_ix_if_i(\bx)}{1+\sum_i c_ix_i}-\dfrac12\dfrac{\sum_{i,j} \sigma_{ij}c_ic_jx_ix_j}{(1+\sum_i c_ix_i)^2}+\gamma_b\left(1+\sum_{i} (|f_i(\bx)|)\right)\right]>0.
\]
As a result Assumption 1.1 of \cite{HN16} is violated and we have to use different methods in this setting.
\end{rmk}

\section{Proofs}\label{s:proof}

\begin{lm}\label{lm3.0}
Suppose that a sequence $\{\nu_k, k=1,2,\dots\}$ of probability measures on $\R^{n}_+$
converges weakly to $\nu_0$.
Furthermore, assume that
$$\sup_k\int_{\R^n_+}\|\bx\|^m\nu_k(d\bx)\leq H. $$
If $h:\R^n_+\mapsto\R$ is a continuous function
satisfying $$\lim_{\|\bx\|\to\infty}\dfrac{h(x)}{\|\bx\|^m}=0$$
then $$\lim_{k\to\infty}\int_{\R^n_+}h(x)\nu_k(d\bx)=\int_{\R^n_+}h(x)\nu_0(d\bx).$$
\end{lm}
\begin{proof}
Let $\eps>0$. Since
$\lim_{\|\bx\|\to\infty}\dfrac{h(x)}{\|\bx\|^m}=0,$
there is $\ell_\eps>0$
such that $|h(x)|\leq \dfrac{\eps\|\bx\|^m}H$ for all $\bx$ satisfying $\|\bx\|>\ell_\eps$.
This implies that for any $k$
$$\int_{\R^n_+}\1_{\{\|\bx\|>\ell_\eps\}}h(x)\nu_k(d\bx)
\leq \dfrac{\eps}H\int_{\R^n_+}\|\bx\|^m\nu_k(d\bx)\leq \eps.$$
Let $\phi_l(\cdot):\R^n_+\to[0,1]$ be a continuous function with compact support satisfying  $\phi_l(\bx)=1$ if $\|\bx\|\leq l_\eps$. 
One gets that for any $k$ the following sequence of inequalities hold
\begin{equation}\label{e2.10}
\begin{aligned}
\int_{\R^n_+}\left(1-\phi_l(\bx)\right)|h(\bx)|\nu_k(d\bx)\leq \int_{\R^n_+}\1_{\{\|\bx\|>\ell_\eps\}}h(x)\nu_k(d\bx)\leq\eps.
\end{aligned}
\end{equation}
Since $\nu_k$ converges weakly to $\nu_0$ we get
\begin{equation}\label{e2.12}
\lim\limits_{k\to\infty}\int_{\R^n_+}\phi_l(\bx)h(\bx)\nu_k(d\bx)=\int_{\R^n_+}\phi_l(\bx)h(\bx)\pi(d\bx).
\end{equation}
As a consequence of \eqref{e2.10} and \eqref{e2.12}
\begin{equation}
\limsup\limits_{k\to\infty}\left|\int_{\R^n_+}h(\bx)\nu_k(d\bx)-\int_{\R^n_+}h(\bx)\pi(d\bx)\right|\leq2\eps.
\end{equation}
The desired result follows by letting $\eps\to0$.

\end{proof}
\begin{lm}\label{lm3.1}
We have the following claims:
\begin{itemize}
\item For any $\bx=(x_1,\dots,x_n)\in \R_+^{n}$ the system \eqref{e:system_2} has a unique strong solution with initial value $\BX(0)=\bx $. The solution satisfies
\begin{equation}\label{e.pos.sol}
\PP_\bx\{X_i(t)>0, t\geq 0\} = 1 \text{ if } x_i>0;\,\text{ and }\, \PP_\bx\{X_i(t)=0, t\geq 0\} = 1\,\text{ if } x_i=0.
\end{equation}
\item The process $\BX$ is a Markov-Feller process.
\item There exist constants $p>0$ and $M_p>0$ such that
\begin{equation}\label{e1-lm3.2}
\limsup_{t\to\infty} \E_\bx \|\BX(t)\|^{1+p}\leq M_p,\, \bx\in\R^n_+.
\end{equation}
\item For any invariant measure $\mu$ of $\BX$, we have
\begin{equation}\label{e2-lm3.2}
\int_{\R^n_+}\|\bx\|^{1+p}\mu(d\bx)\leq M_p.
\end{equation}
\end{itemize}
\end{lm}
\begin{proof}
The existence and uniqueness of strong solutions with initial values $\bx \in \R_+^{n}$
satisfying \eqref{e.pos.sol}
can be shown by standard arguments such as those
from \cite[Theorem 2.1]{LM09} and \cite[Lemma 1]{LB16}. Let $$V(\bx):=\sum_{i=1}^n c_ix_i\,\text{ where }\, c_1=1, c_i:=\prod_{j=2}^i\dfrac{a_{k-1,k}}{a_{k,k-1}}, i\geq 2.$$
We have
\begin{equation}\label{e4-lm3.2}
\begin{aligned}
d V(\BX(t))=&\left[X_1(t)(a_{10}-a_{11}X_1(t))-\sum_{i=2}^{n-1} c_iX_i(t)\left[a_{i0}+a_{i,i+1}X_{i+1}(t)\right]-c_na_{i0}X_n(t)\right]dt\\
&+\sum_{i=1}^n c_iX_i(t)dE_i(t).
\end{aligned}
\end{equation}
If we define $A_1=\min_{i=2}^n\{a_{i0}\}>0$ and $A_0=\sup_{\{x_1>0\}} \left\{x_1(a_{i0}-a_{11}x_1)+A_1x_1\right\}<\infty$ then we can see that for all $\bx\in\R_+^{n,\circ}$
\begin{equation}\label{e4.5-lm3.2}
\left[x_1(a_{10}-a_{11}x_1)-\sum_{i=2}^{n-1} c_ix_i\left[a_{i0}+a_{i,i+1}a_{i+1}x_{i+1}\right]-c_na_{i0}x_n\right] \leq\left[A_0 - A_1 V(\bx)\right].
\end{equation}
Let $p>0$ be sufficiently small such that
\begin{equation}\label{e5-lm3.2}
p\sum_{i,j=1}^n c_ic_jx_ix_j\sigma_{ij}\leq \dfrac{A_1}2V^2(\bx).
\end{equation}
In view of \eqref{e4-lm3.2}, \eqref{e4.5-lm3.2}, \eqref{e5-lm3.2} and It\^o's formula,
we have
\begin{equation}\label{e6-lm3.2}
\begin{split}
\Lom V^{1+p}(\bx)\leq& (1+p)V^p(\bx)[A_0-A_1V(\bx)]+p(1+p)V^{p-1}(\bx) \sum_{i,j=1}^n c_ic_jx_ix_j\sigma_{ij}\\
\leq&(1+p)V^p(\bx)\left[A_0-\dfrac{A_1}2V(\bx)\right] \\
\leq& A_2- A_3 V^{p+1}(\bx) \,\text{ for some } A_2, A_3>0.
\end{split}
\end{equation}
Sine $f_i(\bx)$, $i=1,\dots,n$ are loally Lipschitz funtions,
it follows from \eqref{e6-lm3.2} and \cite[Theorem 5.1]{NYZ17} that
the process $\BX$ is a Feller-Markov process.

Moreover, using \eqref{e6-lm3.2}, Dynkin's formula and a standard argument (see e.g. \cite[Theorem 3.2]{LM09} or \cite[Lemma 2.2]{HN16}),
we can easily obtain that
$$\limsup_{t\to\infty} \E_\bx V^{1+p}(\BX(t))\leq \dfrac{A_2}{A_3}\,\text{ for any }\, \bx\in\R^n_+,$$
which implies \eqref{e1-lm3.2}.
For any invariant probability measure $\mu$ of $\BX$ and $H>0$,
it follows from Fatou's lemma that

$$
\begin{aligned}
\E_\mu \left[H\wedge V^{1+p}(\bx)\right]=&\lim_{t\to\infty}\int_{\R^n_+}\E_\bx\left[H\wedge V^{1+p}(\BX(t)\right]\mu(d\bx)\\
\leq&\int_{\R^n_+}\left(\limsup_{t\to\infty}\E_\bx\left[H\wedge V^{1+p}(\BX(t)\right]\right)\mu(d\bx)\\
\leq & \dfrac{A_2}{A_3}.
\end{aligned}
$$
Then letting $H\to\infty$ we obtain \eqref{e2-lm3.2}.
\end{proof}

\begin{lm}\label{lm3.2}
Suppose $\mu\in\M$ such that $I_\mu= \{n_1,\dots,n_k\}$.
Then
\begin{equation}\label{e:lambda_0}
\lambda_{i}(\mu)=0
\end{equation}
for  any $i\in I_\mu$. As a consequence,
if  $\BX:=(X_1,\dots,X_n)$ has an invariant probability measure $\pi$ on $\R_+^{(j),\circ}$ then
\begin{equation}\label{r:expected_inv}
\E_{\pi}X_i=\int_{\R^n_+} x_i\pi(d\bx)=x_i^{(j)} \,\text{ for } i\leq j.
\end{equation}
That is, the solution of \eqref{e:system} is the vector $(\E_{\pi} X_1,\dots, \E_{\pi} X_j)$ of the expected values of $(X_1,\dots,X_j)$ at stationarity.
\end{lm}
\begin{rmk}
The intuition behind equation \eqref{e:lambda_0} is the following: if we are inside the support of an ergodic invariant measure $\mu$ then we are at an `equilibrium' and the process does not tend to grow or decay.
\end{rmk}
\begin{proof}
Let $\BX^\mu(t)=(X_1^\mu(t),\dots,X_n^\mu(t))$
be the stationary solution
whose distribution at any time $t$ is $\mu$.
By It\^o's formula,
we have
$$
\dfrac{\ln X_i^\mu(t)}t=\dfrac{\ln X_i^\mu(0)}t+\dfrac1t\int_0^t\left[f_i(\BX^\mu(s))-\dfrac{\sigma_{ii}}2\right]ds+\dfrac1t\int_0^t dE_i(s), i\in I_\mu.
$$
Since $f_i$ is a linear function,
it follows from  \eqref{e2-lm3.2} that $f_i $ is $\mu$-measurable.
By the ergodicity of $\BX^\mu(t)$,
$$
\lim_{t\to\infty}\dfrac1t\int_0^t\left[f_i(\BX^\mu(s))-\dfrac{\sigma_{ii}}2\right]ds=\lambda_i(\mu)\, \text{ a.s.}\,, i\in I_\mu.
$$
On the other hand,
it is well-known that almost surely
$$
\lim_{t\to\infty}\dfrac1t\int_0^t dE_i(s)=\lim_{t\to\infty}\frac{E_i(t)}{t}=0,\, i\in I_\mu.
$$
Combining these limits, we obtain
$$
\lim_{t\to\infty}\dfrac{\ln X_i^\mu(t)}t=\lambda_i(\mu)\,\,, i\in I_\mu
$$
almost surely. If $\lambda_i(\mu)$ is nonzero,
we have with probability 1 that
$$
\lim_{t\to\infty} X_i^\mu(t)=\begin{cases}
0&\,\text{ if } \lambda_i(\mu)<0\\
\infty&\,\text{ if } \lambda_i(\mu)>0
\end{cases}\,\text{ for } i\in I_\mu.
$$
Both cases contradict the fact that $\mu(\R^{\mu,\circ}_+)=1.$
The first assertion is therefore proved.

To prove the second claim,
suppose that
$\pi$ is an  invariant probability measure
on $\R^{(j),\circ}_+$. Then
$$\lambda_k(\pi)=
\begin{cases}
\tilde a_{10}-a_{11}\int_{\R^n_+}x_1\pi(d\bx)-a_{12}\int_{\R^n_+}x_2\pi(d\bx)&\text{ if } i=1\\
-\tilde a_{k0}+a_{k,k-1}\int_{\R^n_+}x_{k-1}\pi(d\bx)-a_{k,k+1}\int_{\R^n_+}x_{k+1}\pi(d\bx)&\text{ if } k=2,\dots,j-1\\
-\tilde a_{j0}+a_{j,j-1}\int_{\R^n_+}x_{k-1}\pi(d\bx)&\text{ if } k=j
\end{cases}
$$
Solving the system $\lambda_k(\pi)=0, k=1,\dots,j$
we obtain the desired result.
\end{proof}

\begin{lm}\label{l:inv}
Suppose we have $\mu\in\M$ with $I_\mu= \{n_1,\dots,n_k\}$. Then $I_\mu$ must be of the form $\{1,2,\dots,l\}$ for some $1\leq l\leq n$.
In other words,
for $\mu\in\M$, there exists $l\in\{1,\dots,n\}$ suh that
$\mu(\R^{(l),\circ}_+)=1.$
\end{lm}
\begin{proof}
We argue by contradiction. First, suppose that $n_1>1$. By Lemma \ref{lm3.2}
\begin{equation*}
\begin{split}
\lambda_{n_1}(\mu)= 0 &=-\tilde a_{n_1,0} + a_{n_1,n_1-1}\int_{ R_+^\mu} x_{n_1-1}d\mu \\
&= -\tilde a_{n_1,0} \\
&<0
\end{split}
\end{equation*}
which is a contradiction.

Alternatively, suppose that there exists $\mu\in\M$ such that $I_\mu= \{1,\dots, u^*, v^*,\dots, n_k\}$ with $1\leq u^*<v^*-1\leq n_k\leq n$. As a result one can see that $v^*-1\notin I^\mu$. Then Lemma \ref{lm3.2} leads to
\begin{equation*}
\begin{split}
\lambda_{v^*}(\mu)= 0 &=-\tilde a_{n_1,0} + a_{v^*,v^*-1}\int_{ R_+^\mu} x_{v^*-1}d\mu  - a_{v^*,v^*+1}\int_{ R_+^\mu} x_{v^*+1}d\mu\\
&= -\tilde a_{v^*,0} - a_{v^*,v^*+1}\int_{ R_+^\mu} x_{v^*+1}d\mu\ \\
&<0
\end{split}
\end{equation*}
which is a contradiction.
\end{proof}

\begin{lm}\label{lm3.3}
We have the following assertions.
\begin{enumerate}
\item If $\tilde\kappa(n)>0$ then for any $\pi\in\Conv(\M)$ one has $\max_{i=1}^n\lambda_i(\pi)>0$.
\item If $\tilde\kappa(n)\leq 0$ then for any $\mu\in\M$,
we have $\lambda_n(\mu)\leq0$.
Moreover, there is no invariant probability measure in $\R^{n,\circ}_+$.
\end{enumerate}
\end{lm}
\begin{proof}

If $\tilde\kappa(n)>0$ then $\tilde a_{10}>0$. Thus,
$\lambda_1(\bdelta^*)=\tilde a_{10}>0.$
In view of Lemma \ref{l:inv}, other invariant measures in $\M$ must be supported
on sets of the form $\R^{(i),\circ}_+$ for some $i\in\{1,\dots,n\}$.

Since $\tilde\kappa(i)\geq\tilde\kappa(n)>0$ for $i=1,\dots,n,$
in order to prove claim (1) of the lemma, we need show that if
 $\mu\in\M$ is an invariant measure on $\R^{(i),\circ}_+$
 then $\lambda_{i+1}(\mu)>0.$
By Lemma \ref{lm3.2}
and the definition of $\I_i$,
we have that $\lambda_{i+1}(\mu)=\I_{i+1}$
which, by Lemma \ref{lm4.1}, has the same sign as $\tilde\kappa(i+1)$.
Thus, $\lambda_{i+1}(\mu)>0$.

For $\pi\in\Conv(\M)$,
we can decompose
$\pi=\rho_1\mu_{i_1}+\dots+\rho_k\mu_{i_k}$
where
$0\leq i_1<\dots<i_k<n$, $\rho_j>0$ for $j=1,\dots,k$
and $\mu_{i_j}\in\M_{i_j}$.
Since $i_j\geq i_1+1$ for $j=2,\dots,k$
we have from Lemma \ref{lm3.2} that $\lambda_{i_1+1}(\mu_{i_k})=0$.
Thus,
$\lambda_{i_1+1}(\pi)=\rho_1\lambda_{i_1+1}(\mu_{i_1})=\rho_1\I_{i+1}>0$.
Part (i) is therefore proved.

Now, suppose that $\tilde\kappa(n)\leq 0$.
Clearly, if $\mu\in\M$ and $\mu(\R^{(j),\circ}_+)=1$ for $j<n+1$,
then $\lambda_n(\mu)=-\tilde a_{n0}<0$.
If $\mu(\R^{(n-1),\circ}_+)=1$, it follows from \eqref{e:inv_j+1}, \eqref{Lya.exp} and \eqref{r:expected_inv} that
$\lambda_n(\mu)=\I_n\leq 0$
since $\I_n$ has  the same sign as $\tilde\kappa(i+1)$.

Finally, if there is an invariant  probability measure $\pi$ in $\R^{n,\circ}_+$
then $x^{(n)}_n=\int_{\R^n_+}x_n\pi(d\bx)>0$
is the $n$-th component of the solution to \eqref{e:system} with $j$ replaced by $n$.
By Lemma \ref{lm4.1}, we must have $\tilde\kappa(n)>0$,
which completes the proof.
\end{proof}

\begin{lm}\label{lm3.4}
For $i=1,\dots,n$, let $\M_i$ be the set of invariant probability measures $\mu$
of $\BX$ satisfying $\mu\left(\R^{(i),\circ}_+\right)=1$.
In particular, let $\M_0=\{\bdelta^*\}$.
Suppose that $\tilde\kappa_j>0$ for some $j\in\{1,\dots,n\}$.
Then, if $\M_j\neq \emptyset$, the family $\M_j$
is tight in $\R^{(j),\circ}_+$.
\end{lm}
\begin{proof}
Suppose that $\M_j\neq \emptyset$.
Then it is tight in $\R^{(j)}_+$
due to \eqref{e2-lm3.2}.
Suppose that $\M_j$ is not tight in $\R^{(j),\circ}_+$. Then we can extract a sequence of probability measure $\{\mu_k,k\in\N\}\subset\M_j$,
such that
$\mu_k$ converges weakly to
a probability measure $\tilde\mu$
satisfying
$\mu\left(\partial \R^{(j)}_+\right)>0$.
Decompose $\tilde\mu=\tilde p\tilde\mu_1+(1-\tilde p)\tilde\mu_2$
where $\tilde p\in(0,1]$
and $\tilde\mu_1\left(\partial \R^{(j)}_+\right)=1, \tilde\mu_2\left(\R^{(j),\circ}_+\right)=1$.
Since the $\mu_k$'s are invariant probability measures of $\BX$
and $\partial \R^{(j)}_+$ and $\R^{(j),\circ}_+$ are invariant sets,
we deduce that
$\tilde\mu_1$ and $\tilde\mu_2$
are also invariant probability measures of $\BX$.
Using part (1) of Lemma \ref{lm3.3} with $n$ replaced by $j$,
we have
$$\max_{i=1,\dots,j}\lambda_i(\tilde\mu_1)>0.$$
In view of Lemma \ref{lm3.2},
\begin{equation}\label{e1-lm3.4}
\lambda_i(\mu)=0\,\text{ for }\,i=1,\dots,j\, \text{ if }\mu\in\M_j.
\end{equation}
In particular $\lambda_i(\tilde\mu_2)=0$ for $i=1,\dots,j$.
Thus,
\begin{equation}\label{e2-lm3.4}
\max_{i=1,\dots,j}\lambda_i(\tilde\mu)>0.
\end{equation}
On the other hand,
it follows from \eqref{e1-lm3.4}, \eqref{e2-lm3.2} and Lemma \ref{lm3.0} that
$$
\max_{i=1,\dots,j}\lambda_i(\tilde\mu)=\lim_{k\to\infty}\max_{i=1,\dots,j}\lambda_i(\mu_k)=0
$$
which contradicts \eqref{e2-lm3.4}.
As a result, $\M_j$ is tight in $\R^{(j),\circ}_+$.
\end{proof}

\begin{proof} [Proof of Theorem \ref{t:stoc} (i)]
In view of It\^o's formula,
\begin{equation}\label{e1-thm1.3}
\begin{split}
\ln X_1(t) &=\ln X_1(0)+\int_0^t(\tilde a_{10} - a_{11}X_1(s) - a_{12}X_2(s))\,ds + E_1(t)\\
\ln X_2(t) &=\ln X_2(0)+\int_0^t(-\tilde a_{20} + a_{21}X_1(s) - a_{23}X_3(s))\,ds\,+E_2(t)\\
&\mathrel{\makebox[\widthof{=}]{\vdots}} \\
\ln X_{n-1}(t) &=\ln X_{n-1}(0)+\int_0^t(-\tilde a_{n-1,0}+a_{n-1,n-2}X_{n-2}(s) - a_{n-1,n}X_n)\,ds+E_{n-1}(t)\\
\ln X_n(t) &=\ln X_n(0)+\int_0^t(-\tilde a_{n0} + a_{n,n-1}X_{n-1}(s)\,dt + E_n(t).
\end{split}
\end{equation}
Since the expectations with respect to $\E_\bx$ of the right hand side terms exist,
the expectations $\E_\bx\ln X_i(t)$, $i=1,\dots,n$ also exist.
As a result of Lemma \ref{lm3.1} for all $\bx\in\R^n_+$ we have
\begin{equation}\label{e2-thm1.3}
\limsup_{t\to\infty} \dfrac{\E_\bx\ln X_i(t)}{t}\leq 0, \,i=1,\dots,n.
\end{equation}
Consider the empirical measure
$$\Pi^\bx_t=\frac{1}{t}\int_0^t\PP_\bx\{\BX(s)\in\cdot\}ds.$$
In view of Lemma \ref{lm3.2}, for each $\bx\in\R^n_+$, the family $\{\Pi^\bx_t,t\geq0\}$ is tight. The weak limit points of $\{\Pi^\bx_t,t\geq0\}$ are invariant probability measures of $\BX$ (see e.g. \cite[Proposition 6.4]{EHS15}).
Let $\nu$ be any  weak limit point of $\{\Pi^\bx_t,t\geq0\}$.
We decompose
$\nu$ as
$\nu=\rho\nu_1+(1-\rho)\nu_2$
where
$\rho\in(0,1)$, $\nu_1\in\Conv(\M)$
and $\nu_2\in\M_n$.
At this moment, we have not proved that
$\M_n$ is nonempty,
but the decomposition is well-defined
because we can let $\rho=1$ if $\M_n$ is empty.
Suppose that we have a sequence $(t_\ell,\ell\in\N)$ with $t_\ell\to\infty$ and that $\left(\Pi^\bx_{t_\ell}\right)_\ell$  converges weakly to the probability measure $\nu$.
Moreover, using \eqref{e1-lm3.2} and Lemma \ref{lm3.0}, Lemma \ref{lm3.1},
we can show that
$$\lim_{\ell\to\infty}\int_{\R^n_+}x_i\Pi^\bx_{t_\ell}(d\bx)=\int_{\R^n_+}x_i\nu(d\bx).$$
In light of Lemma \ref{l:inv}, we can write $\nu_1$ in the form
$\nu_1=q_1\mu_{k_1}+\dots+q_m\mu_{k_1}$
where $\mu_{k_i}\in\M_i, q_i>0$ for $i=1,\dots,m$,
and
$0\leq k_1<\dots<k_m<n$.

We have proved that $\lambda_{k_1+1}(\mu_{k_1})>0$.
If $m\geq 2$ then
we have
from Lemma \ref{lm3.2} that $\lambda_{j}(\mu_{k_i})=0$ for $i=2,\dots, m$ and $j=1,\dots,k_i$.
In particular, $\lambda_{k_1+1}(\mu_{k_i})=0$ for $i=2,\dots,m$.
Moreover $\lambda_i(\nu_2)=0$ for $i=1,\dots,n$.
Thus $\lambda_{k_1+1}(\nu)=\rho\lambda_{k_1+1}(\nu_1)=\rho q_1\lambda_{k_1+1}(\mu_{k_1})$.
Therefore
$$
\begin{aligned}
\lim_{\ell\to\infty}&  \dfrac{\E_\bx\ln X_{k_1+1}(t_\ell)}{t}\\
=
&
\lim_{\ell\to\infty}\left[\dfrac{\ln x_{k_1+1}}{t_\ell}+\dfrac1{t_\ell}\E_\bx\int_0^{t_\ell}(-\tilde a_{k_10} + a_{k_1,k_1-1}X_{k_1-1}(s)+a_{k_1,k_1+1}X_{k_1+1}(s)\,dt\right]\\
=&\lim_{\ell\to\infty}\int_{\R^n_+}\left[-\tilde a_{k_10} + a_{k_1,k_1-1}x_{k_1-1}+a_{k_1,k_1+1}x_{k_1+1}\right]\Pi^\bx_{t_\ell}(d\bx)\\
=&\lambda_{k_1+1}(\nu)=\rho q_1\lambda_{k_1+1}(\mu_{k_1})\geq 0.
\end{aligned}
$$
This and \eqref{e2-thm1.3} imply that
$\rho=0$.
Thus, $\M_n$ is nonempty
and all the weak limit points of $\{\Pi^\bx_t,t\geq0\}$
belong to $\M_n$.

Moreover, since $\M_n$ is tight in $\R^{n,\circ}_+$ (see Lemma \ref{lm3.4}),
we can find for any $\eps>0$ a compact set $K_\eps\subset \R^{n,\circ}_+$
such that
$$\liminf_{t\to\infty}\Pi^\bx_t(K_\eps)=\liminf_{t\to\infty}\dfrac1t\int_0^t\PP_\bx\{\BX(s)\in K_\eps\}ds\geq 1-\eps.$$

If $\Gamma$ is positive definite, the diffusion process $\BX$ is nondegenerate in $\R^{n,\circ}_+$.
As a result the invariant probability measure on $\R^{n,\circ}_+$ is unique
and the strong stochastic persistence of $\BX$
follows from \cite[Theorem 20.20]{K02}.
\end{proof}

\begin{proof} [Proof of Theorem \ref{t:stoc} (ii)]

Suppose that $\tilde\kappa(j^*)>0$ and $\tilde\kappa(j^*+1)\leq0$.
By part (2) of Lemma \ref{lm3.3},
if $\mu$ is an ergodic measure of $\BX$
then $\lambda_j(\mu)<0$ for $j>j^*$.
Applying Lemma \ref{lm3.2} and Lemma \ref{lm3.3},
we deduce that
\begin{equation}\label{e.mu.j}
\mu\left(\R^{(j^*)}_+\right)=1
\text{ for any }\mu\in\M.
\end{equation}

Suppose $\bx\in\R^{n,\circ}_+$.
Let $\{t_k\}$ be any subsequence such that $\lim_{k\to\infty} t_k=\infty$
and $\Pi^\bx_{t_k}(d\bx)$ converges weakly to some measure $\nu\in\Conv(\M)$.
We decompose
$\nu$ as
$\nu=\rho\nu_1+(1-\rho)\nu_2$
where
$\rho\in(0,1)$, $\nu_1$ is a invariant probability measure on $\partial\R^{(j^*),\circ}_+$
and $\nu_2\in\M_{j^*}$.
Then, using the arguments from the proof for part (i) of Theorem \ref{t:stoc}
with $n$ replaced by $j^*$,
we can show that
any weak limit points of $\left(\Pi^\bx_{t_k}(d\bx)\right)_{k}$ for $\bx\in\R^{n,\circ}_+$ belong to $\M_{j^*}$.
The time-average persistence in probability of $(X_1,\dots, X_{j^*})$
then follows from Lemma \ref{lm3.4}.
Moreover,
since $$\int_{\R^n_+}x_i'\mu(d\bx')=
\begin{cases}
x^{(j^*)}_i\,&\text{ if } i=1,\dots,j^*,\\
0\,&\text{ if } i=j^*+1,\dots, n.
\end{cases}
\,\text{ for }\, \mu\in\M_{j^*},
$$
for $\bx\in\R^{n,\circ}_+$,
we have from Lemma \ref{lm3.0} and Lemma \ref{lm3.2} that $$
\lim_{t\to\infty}\dfrac1t\int_0^t\E_\bx X_i(s)ds=
\begin{cases}
x^{(j^*)}_i\,&\text{ if } i=1,\dots,j^*,\\
0\,&\text{ if } i=j^*+1,\dots, n.
\end{cases}
$$

\end{proof}

\begin{proof}[Proof of Theorem \ref{t:stoc} (iii)]

Consider the system
\begin{equation}\label{sde:2d}
\begin{split}
dX_1(t) &= X_1(t)(a_{10} - a_{11}X_1(t) - a_{12}X_2(t))\,dt + X_1(t)\,dE_1(t)\\
dX_2(t) &= X_2(t)(-a_{20} + a_{21}X_1(t))\,dt+X_2(t)\,dE_2(t)
\end{split}
\end{equation}
Define the process $\tilde X$ via
\begin{equation}\label{sde:tilde1}
d\tilde X(t) = \tilde X(t)(a_{10} - a_{11}\tilde X(t) )\,dt + \tilde X(t)\,dE_1(t).
\end{equation}

By a comparison argument (see \cite{HNY16,GM94} and \cite{EHS15}[Theorem 6.1]), if $X_1(0)=\tilde X(0)=x\in\R_+^\circ$ then
\begin{equation}\label{e:comp}
\PP_x\left\{X_1(t)\leq \tilde X(t), t\geq0\right\}=1, x\in\R_+^\circ.
\end{equation}
The symptotic behavior of $\tilde X$ is wel-known (see \cite{EHS15}). Namely,
\begin{itemize}
\item If $\I_1=\tilde a_{10}<0$ then $\PP_x\left\{\lim_{t\to\infty}\tilde X(t)=0\right\}=1, x\in\R_+^\circ.$
\item  If $\I_1=\tilde a_{10}=0$ then $\PP_x\left\{\lim_{t\to\infty}\dfrac1t\int_0^t\tilde X(s)ds=0\right\}=1, x\in\R_+^\circ.$
\item If $\I_1=\tilde a_{10}>0$ then $\PP_x\left\{\lim_{t\to\infty}\dfrac1t\int_0^t\tilde X(s)ds=\dfrac{\tilde a_{10}}{a_{11}}\right\}=1.$
\end{itemize}
By It\^o's formula, equation \eqref{e:comp} and the asymptotic behavior of $\dfrac1t\int_0^t\tilde X(s)ds$ we have that if  $(X_1(0),X_2(0))=\bx \in\R_+^\circ$ then
\begin{equation}\label{lnX2}
\begin{split}
\limsup_{t\to\infty}\dfrac{\ln X_2(t)}t=&\limsup_{t\to\infty}\left[-\tilde a_{20}+a_{21}\dfrac1t\int_0^tX_1(s)ds+\dfrac{E_2(t)}t\right]\\
\leq& \limsup_{t\to\infty}\left[-\tilde a_{20}+a_{21}\dfrac1t\int_0^t\tilde X(s)ds\right]\\
\leq& -\tilde a_{20}+a_{21}\left[\dfrac{\tilde a_{10}}{a_{11}}\right]\vee 0\,\,\,\,\PP_\bx-\text{ a.s.}
\end{split}
\end{equation}
Thus, if $\I_1\leq 0$ or if $\I_2=-\tilde a_{20}+a_{21}\left[\dfrac{\tilde a_{10}}{a_{11}}\right]<0$ then $X_2(t)$ goes to 0 as $t\to\infty$ almost surely with respect to $\PP_x, x\in\R_+^\circ$.
This and the fact that $\limsup_{t\to\infty}\frac{1}{t}\int_0^t X_1(s)ds\leq \left[\frac{\tilde a_{10}}{a_{11}}\right]\vee 0$
imply that the family of random nomalized occupation measures
$$\tilde\Pi_t(\cdot):=\dfrac1t\int_0^t\1_{\{\BX(s)\in\cdot\}}\,ds,t\geq0$$
is tight
and the family of its weak$^*$-limits when $t\to\infty$
is a subset of $\Conv(\M)$ (see e.g.\cite{EHS15}).
Using the fact (see  \cite{DNY16} or \cite{HNY16}) that
$$\limsup_{t\to\infty}\dfrac1t\int_0^t X_1^2(s)ds\leq \limsup_{t\to\infty}\dfrac1t\int_0^t\tilde X^2(s)ds<\infty$$
and Lemma \ref{lm3.0},
we deduce that if $(t_k)_{k\in\N}\subset\R_+$ is a random sequence going to $\infty$ as $k\to\infty$ and $\tilde\Pi_{t_k}(\cdot)$ converges weakly to a random measure $\pi$
almost surely, then
\begin{equation}\label{e.ui}
\lim_{k\to\infty}\dfrac1{t_k}\int_0^{t_k}X_1(s)ds\to \int_{\R^n_+}x_1\pi(dx_1,dx_2).
\end{equation}

For $(X_1(0),X_2(0))=\bx \in\R_+^{2,\circ}$, consider two cases.
\begin{itemize}
\item If $\I_1\leq 0$ then
by Theorem \ref{t:stoc}, part (ii), $\bdelta^*$
is the unique invariant probability measure of $(X_1(t),X_2(t))$.
We have from \eqref{e.ui} that
as $t\to\infty$ almost surely with respect to  $\PP_\bx$ the occupation measures $\tilde\Pi_t(\cdot)$ converge weakly to $\bdelta^*$ and
$$\PP_\bx\left\{\lim_{t\to\infty}\dfrac{\ln X_1(t)}t=\I_1=\tilde a_{10}\,\text{ and }\,\lim_{t\to\infty}\dfrac{\ln X_2(t)}t=-\tilde a_{20}\right\}=1.$$
\item If $\I_1>0$ and $\I_2<0$,
using \eqref{lnX2}, (which leads to $\limsup_{t\to\infty}\frac{\ln X_2(t)}t\leq0$ a.s), \eqref{e.ui} and a contradiction argument that is similar to the one from the proof of Theorem \ref{t:stoc}, part (i) and (ii), we get that
as $t\to\infty$ almost surely with respect to  $\PP_\bx$, the occupation measures $\tilde\Pi_t(\cdot)$ converge weakly to the unique invariant probability measure on $\R^{(1),\circ}_+$   and
$$\PP_\bx\left\{\lim_{t\to\infty}\dfrac{\ln X_2(t)}t=\I_2\right\}=1.$$
\end{itemize}
\end{proof}
\section{Invasion rates}\label{s:inv}
We want to analyze the invasion rates $\I_{i}, i=1,\dots,n+1$. For this we note by \eqref{e:inv_j+1} that we have to analyze the system \eqref{e:system}. This can be written in matrix form as
\begin{equation}\label{e:inv_mat}
A \bx^{(n)} = \ba
\end{equation}
where $\bx^{(n)} = \left(x_1^{(n)},\dots,x_n^{(n)}\right)^T$, $\ba = (-\tilde a_{10},\tilde a_{20},\tilde a_{30},\dots, \tilde a_{n0})^T$ and
\[
A= \begin{bmatrix}
    -a_{11} & -a_{12} & 0 &\dots  & 0 &0 \\
    a_{21} & 0 & -a_{23}&\dots  & 0 &0 \\
    0 & a_{32} & 0&\dots  & 0 &0 \\
    \vdots & \vdots & \vdots & \ddots & \vdots & \vdots\\
    0 & 0 & 0 & \dots  &0&-a_{n-1,n}\\
     0 & 0 & 0 & \dots  &a_{n,n-1}&0
\end{bmatrix}
\]
It is well-known that the solution can be obtained by a forward sweep that is a special case of Gaussian elimination (see \cite{M01}). To simplify notation we let
\[
(d_1,\dots,d_n)^T :=(-\tilde a_{10},\tilde a_{20},\tilde a_{30},\dots, \tilde a_{n0})^T,
\]
\[
(c_1,\dots,c_{n-1})^T := (-a_{12},-a_{23},\dots,-a_{n-1,n})^T,
\]

and
\[
(f_2,\dots,f_n)^T := (a_{21}, a_{32},\dots,a_{n,n-1})^T.
\]
Define new coefficients $(c_1',\dots,c_{j-1}'), (d_1',\dots,d_j')$ recursively as
\begin{equation}\label{e:c'_2}
c_i'=\left\{
	\begin{array}{ll}
		\frac{c_1}{-a_{11}}, &  i=1\\
		\frac{c_i}{-f_ic_{i-1}'}, & i=2,3,\dots,n-1
	\end{array}
\right.
\end{equation}
and
\begin{equation}\label{e:d'_2}
d_i'=\left\{
	\begin{array}{ll}
		\frac{d_1}{-a_{11}}, &  i=1\\
		\frac{d_i-f_id_{i-1}'}{-f_ic_{i-1}'}, & i=2,3,\dots,n.
	\end{array}
\right.
\end{equation}
Let $l_i:=\frac{c_i}{-f_i}$. Then \eqref{e:c'_2} can be written as
\[
\frac{c'_{j}}{c'_{j-2}}=\frac{l_j}{l_{j-1}}
\]
which yields when $j$ is even and $4\leq j\leq n$,
\begin{equation}\label{e:c'_even}
\begin{split}
c'_{j}&=c_2' \prod_{i=2}^{j/2}\frac{l_{2i}}{l_{2i-1}}\\
&= -\frac{c_2b_1}{f_2 c_1}\prod_{i=2}^{j/2}\frac{\frac{c_{2i}}{-f_{2i}}}{\frac{c_{2i-1}}{-f_{2i-1}}}\\
&=  \frac{a_{23}a_{11}}{a_{21}a_{12}}\prod_{i=2}^{j/2} \frac{a_{2i-1,2i}}{a_{2i-2,2i-1}}\frac{a_{2i-1,2i-2}}{a_{2i,2i-1}}
\end{split}
\end{equation}
while if $j$ is odd and $3\leq j\leq n$,
\begin{equation}\label{e:c'_odd}
\begin{split}
c'_j&= c'_{1}\prod_{i=1}^{(j-1)/2}\frac{l_{2i+1}}{l_{2i}}\\
&=\frac{c_1}{b_1} \prod_{i=1}^{(j-1)/2}\frac{\frac{c_{2i+1}}{-f_{2i+1}}}{\frac{c_{2i}}{-f_{2i}}}\\
&=\frac{a_{12}}{a_{11}} \prod_{i=1}^{(j-1)/2} \frac{a_{2i,2i+1}}{a_{2i-1,2i}}\frac{a_{2i,2i-1}}{a_{2i+1,2i}}.
\end{split}
\end{equation}
Next, we want to find a formula for $d_n'$. Set $e_i:=\frac{d_{i+1}}{-f_{i+1}c'_{i}}$ and $g_i:=\frac{1}{c'_{i}}$. Then equation \eqref{e:d'_2}can be written as
\[
d_{i+1}' = e_i+g_id_i'.
\]
One can see that the formula for this recursion is given by
\begin{equation}\label{e:d'_form}
\begin{split}
d_n' &= \sum_{j=1}^{n-1}e_j\prod_{i=j+1}^{n-1}g_i + d_1'\prod_{i=1}^{n-1}g_i\\
&= \sum_{j=1}^{n-1}\frac{d_{j+1}}{-f_{j+1}c'_{j}}\prod_{i=j+1}^{n-1}\frac{1}{c'_{i}} + \frac{d_1}{b_1}\prod_{i=1}^{n-1}\frac{1}{c'_{i}}\\
&=-\sum_{j=1}^{n-1} \frac{\tilde a_{j+1,0}}{a_{j+1,j}}  \prod_{i=j}^{n-1}\frac{1}{c'_{i}} + \frac{\tilde a_{10}}{a_{11}}\prod_{i=1}^{n-1}\frac{1}{c'_{i}}\\
&=\prod_{i=1}^{n-1} \frac{1}{c_i}\left(\sum_{j=1}^{n-1}\frac{d_{j+1}}{-f_{j+1}}\prod_{i=1}^{j-1}c_i' + d_1'\right)
\end{split}
\end{equation}
where we make the convention that $\prod_{i=l}^m g_i = 1$ if $l>m$.
For $n\geq 3$ is odd we get
\begin{equation}\label{e:d'_odd}
\begin{split}
d_n' &= \left(\prod_{i=1}^{(n-1)/2}\frac{a_{2i,2i-1}}{a_{2i-1,2i}}\right)\left( \frac{\tilde a_{10}}{a_{11}} -\sum_{j=1}^{(n-1)/2} \frac{\tilde a_{2j+1,0}}{a_{2j+1,2j}} \frac{a_{12}}{a_{11}} \prod_{i=2}^j \frac{a_{2i-2,2i-1}}{a_{2i-1,2i-2}} -\sum_{j=1}^{(n-1)/2} \frac{\tilde a_{2j,0}}{a_{2j,2j-1}} \prod_{i=1}^{j-1}\frac{a_{2i-1,2i}}{a_{2i,2i-1}}\right)
\end{split}
\end{equation}
while for $n\geq 4$ even
\begin{equation}\label{e:d'_even}
\begin{split}
d_n' &= \left(\frac{a_{11}}{a_{12}}\prod_{i=1}^{(n-2)/2}\frac{a_{2i+1,2i}}{a_{2i,2i+1}}\right)\left( \frac{\tilde a_{10}}{a_{11}} -\sum_{j=1}^{(n-2)/2} \frac{\tilde a_{2j+1,0}}{a_{2j+1,2j}} \frac{a_{12}}{a_{11}} \prod_{i=2}^j \frac{a_{2i-2,2i-1}}{a_{2i-1,2i-2}} -\sum_{j=1}^{n/2} \frac{\tilde a_{2j,0}}{a_{2j,2j-1}} \prod_{i=1}^{j-1}\frac{a_{2i-1,2i}}{a_{2i,2i-1}}\right)
\end{split}
\end{equation}
As a result, if $n\geq 3$ is odd we have
\begin{equation}\label{e:inv_j+1_expr_odd}
\begin{split}
\I_{n+1} =& -\tilde a_{n+1,0} + a_{n+1,n} x^{(n)}_n\\
=& -\tilde a_{n+1,0} + a_{n+1,n} d'_n\\
=&  -\tilde a_{n+1,0} + a_{n+1,n}  \left(\prod_{i=1}^{(n-1)/2}\frac{a_{2i,2i-1}}{a_{2i-1,2i}}\right) \Bigg(\frac{\tilde a_{10}}{a_{11}} -\sum_{j=1}^{(n-1)/2} \frac{\tilde a_{2j+1,0}}{a_{2j+1,2j}} \frac{a_{12}}{a_{11}} \prod_{i=2}^j \frac{a_{2i-2,2i-1}}{a_{2i-1,2i-2}}\\
&-\sum_{j=1}^{(n-1)/2} \frac{\tilde a_{2j,0}}{a_{2j,2j-1}} \prod_{i=1}^{j-1}\frac{a_{2i-1,2i}}{a_{2i,2i-1}}\Bigg)
\end{split}
\end{equation}
while for $n\geq 4$ even
\begin{equation}\label{e:inv_j+1_expr_even}
\begin{split}
\I_{n+1} =& -\tilde a_{n+1,0} + a_{n+1,n} x^{(n)}_n\\
=& -\tilde a_{n+1,0} + a_{n+1,n} d'_n\\
=&  -\tilde a_{n+1,0} + a_{n+1,n}  \left(\frac{a_{11}}{a_{12}} \prod_{i=1}^{(n-2)/2}\frac{a_{2i+1,2i}}{a_{2i,2i+1}}\right) \Bigg(\frac{\tilde a_{10}}{a_{11}} -\sum_{j=1}^{(n-2)/2} \frac{\tilde a_{2j+1,0}}{a_{2j+1,2j}} \frac{a_{12}}{a_{11}} \prod_{i=2}^j \frac{a_{2i-2,2i-1}}{a_{2i-1,2i-2}}\\
&-\sum_{j=1}^{n/2} \frac{\tilde a_{2j,0}}{a_{2j,2j-1}} \prod_{i=1}^{j-1}\frac{a_{2i-1,2i}}{a_{2i,2i-1}}\Bigg).
\end{split}
\end{equation}
\begin{lm}\label{lm4.1}
The quantities $x_n^{(n)}$, $\I_n$ and $\tilde\kappa(n)$
have the same sign for any $n\geq 1$.
\end{lm}
\begin{proof}
It follows from \eqref{e:d'_odd}, \eqref{e:d'_even} and \eqref{e:tkappa} that
$$
x_n^{(n)}=d_n'=\begin{cases}
\dfrac1{a_{11}}\left(\prod_{i=1}^{(n-1)/2}\frac{a_{2i,2i-1}}{a_{2i-1,2i}}\right)\tilde\kappa(n)&\,\text{ if } n \text{ is odd}\\
\frac{1}{a_{12}}\left(\prod_{i=1}^{(n-2)/2}\frac{a_{2i+1,2i}}{a_{2i,2i+1}}\right)\tilde\kappa(n)&\,\text{ if } n \text{ is even.}
\end{cases}
$$
Thus, $x_n^{(n)}$ and $\tilde\kappa(n)$ have the same sign.
On the other hand, we deduce from \eqref{e:d'_2} that
$$
x_n^{(n)}=d_n'=\frac{d_n-f_nd_{n-1}'}{-f_nc_{n-1}'}=\dfrac{\I_{n}}{-f_nc_{n-1}'}, n\geq 2.
$$
In view of \eqref{e:c'_2},
we have
$c_{n-1}'<0.$
Thus, $x_n^{(n)}$ and $\I_{n}$ have the same sign for $n\geq 2$.
The case $n=1$ is trivial.
\end{proof}
\begin{rmk}
The invasion rates are functions of the variances $(\sigma_{ii})_{i=1,\dots,n}$. We note that $\I_j(\sigma_1,\dots,\sigma_j)$ is \textbf{strictly decreasing} in each variable $\sigma_i, 1\leq i\leq j$. As a result environmental stochasticity is seen to increase the risk of extinction.

In the limit of no noise (i.e. $\sigma_{ii}\downarrow 0$ for $1\leq 1\leq n$) the invasion rates converge to $\hat \I_i$, that is  $\I_i\uparrow \hat \I_i$ as $\sigma_{ii}\downarrow 0$,  where

\begin{equation}\label{e:inv_j+1_expr_odd_2}
\begin{split}
\hat \I_{n+1} =&  - a_{n+1,0} + a_{n+1,n}  \left(\prod_{i=1}^{(n-1)/2}\frac{a_{2i,2i-1}}{a_{2i-1,2i}}\right) \Bigg(\frac{ a_{10}}{a_{11}} -\sum_{j=1}^{(n-1)/2} \frac{ a_{2j+1,0}}{a_{2j+1,2j}} \frac{a_{12}}{a_{11}} \prod_{i=2}^j \frac{a_{2i-2,2i-1}}{a_{2i-1,2i-2}}\\
&-\sum_{j=1}^{(n-1)/2} \frac{ a_{2j,0}}{a_{2j,2j-1}} \prod_{i=1}^{j-1}\frac{a_{2i-1,2i}}{a_{2i,2i-1}}\Bigg)
\end{split}
\end{equation}
while for $n\geq 4$ even
\begin{equation}\label{e:inv_j+1_expr_even_2}
\begin{split}
\hat \I_{n+1} =&  - a_{n+1,0} + a_{n+1,n}  \left(\frac{a_{11}}{a_{12}} \prod_{i=1}^{(n-2)/2}\frac{a_{2i+1,2i}}{a_{2i,2i+1}}\right) \Bigg(\frac{ a_{10}}{a_{11}} -\sum_{j=1}^{(n-2)/2} \frac{ a_{2j+1,0}}{a_{2j+1,2j}} \frac{a_{12}}{a_{11}} \prod_{i=2}^j \frac{a_{2i-2,2i-1}}{a_{2i-1,2i-2}}\\
&-\sum_{j=1}^{n/2} \frac{ a_{2j,0}}{a_{2j,2j-1}} \prod_{i=1}^{j-1}\frac{a_{2i-1,2i}}{a_{2i,2i-1}}\Bigg).
\end{split}
\end{equation}
Even though our methods do not work in the deterministic setting, the expressions for $\hat \I_1,\dots,\hat \I_n$ give, correctly, the deterministic invasion rates.
\end{rmk}

\section{Discussion}\label{s:disc}
We have analysed the persistence and extinction of species decscribed by a stochastic Lotka-Volterra food chain. Our main result, Theorem \ref{t:stoc}, looks at the setting when there is no intracompetition for the predator species and any species can only interact with the species that are directly above or below it  in the food chain. We show that similarly to the deterministic case (see \cite{GH79}) one single factor, $\tilde\kappa(n)$, determines which species persist and which go extinct in a weak sense. It is interesting to note that one can recover $\tilde\kappa(n)$ from the constant $\kappa(n)$ (which determines the behavior of the deterministic food-chain), by doing the substitutions $a_{10}\mapsto \tilde a_{10} =a_{10} - \frac{\sigma_{11}}{2}$ and $a_{j0}\mapsto \tilde a_{j0}= a_{j0} + \frac{\sigma_{jj}}{2}, j\geq 2$. Therefore, from a persistence/extinction point of view, the effect of the stochastic environment is that it lowers the growth rate of the prey species by one half of the variance of the noise affecting the prey and increases the death rates of all the predators by one half of the variance of the respective noise terms. This shows that in this model environmental noise inhibits the coexistence of species.

For technical reasons we cannot say anything about the speed of convergence to the invariant probability measure. When one or more species go extinct we can only show that they go extinct in a weak sense (other than in dimension $n=2$, where we prove stronger results).

In Section \ref{s:inv} we give explicit expressions for the invasion rates. The invasion rates are closely related to the factors $(\tilde\kappa(i), i=1,\dots,n)$ - something that is shown in Lemma \ref{lm4.1}.
Our results generalize the results from the deterministic setting of \cite{GH79} to their natural stochastic analogues. We are able to find an algebraically tractable criterion (just like in the deterministic setting) for persistence and extinction.

The invasion rates are shown to be closely related to the first moments of the invariant measures living on the boundary $\partial\R_+^n$ of the system. This is the analogue of looking for the different equilibrium points of the deterministic system \eqref{e:det} and then studying the stability of these points.

The main simplification of our model is the fact that the dynamics of each trophic level is governed by the adjoining trophic levels which immediately precede or succeed it. This fact makes it possible to explicitly describe the structure of the ergodic invariant probability measures of the system living on the boundary $\partial\R_+^n$ (Lemma \ref{l:inv}). The key property of an invariant probability measure $\mu$ living on $\partial\R_+^n$ is that if predator $X_j$ is not present then all predators that are above $j$ (that is, $X_i$ with $i>j$) are also not present. This fact is biologically clear because if species $X_j$ does not exist then $X_{j+1}$ must go extinct since it does not have a food source.

We show that the introduction of a new top predator into the ecosystem makes extinction more likely. This agrees with the deterministic case studied in \cite{GH79}.

For more complex interactions between predators and their prey (i.e. a food web instead of a food chain), even when $n=3$, the possible outcomes become much more complicated. We refer the reader to \cite{HN16} for a detailed discussion of the case when one has one prey and two predators and the apex predator eats both the intermediate predator and the prey.

Of course, one would usually want to prove some stronger results for when the species go extinct and $n>2$. We conjecture the following result holds.

\begin{conj}
Assume that $a_{11}>0$, $\Sigma$ is positive definite and $\BX(0)=\bx\in\R_+^{n,\circ}$.  If there exists $j^*<n$ such that $\tilde \kappa(j^*)>0$ and $\tilde \kappa(j^*+1)<0$ then the predators $(X_{j^*+1},\dots,X_n)$ go extinct, that is
\[
\PP_x\left\{\lim_{t\to\infty}\frac{\ln X_k(t)}{t}=\tilde a_{k0}\right\}= 1, k>j^*.
\]
At the same time, the normalized occupation measure of $(X_1,\dots,X_{j^*})$ converges weakly to the unique invariant probability measure $\pi^{(j^*)}$ on $\R_+^{(j^*),\circ}$.
\end{conj}

We are able to prove this conjecture when there exists strictly positive intraspecies competition among the predators. These results will appear in the follow-up paper \cite{HN17b}.

\cite{GH79, MHP14} are able to analyze the system
\begin{equation}\label{e:det2}
\begin{split}
dx_1(t) &= x_1(t)(a_{10} - a_{12}x_2(t))\,dt\\
dx_2(t) &= x_2(t)(-a_{20} + a_{21}x_1(t) - a_{23}x_3(t))\,dt\\
&\mathrel{\makebox[\widthof{=}]{\vdots}} \\
dx_{n-1}(t) &= x_{n-1}(t)(-a_{n-1,0}+a_{n-1,n-2}x_{n-2}(t) - a_{n-1,n}x_n)\,dt\\
dx_n(t) &= x_n(t)(-a_{n0} + a_{n,n-1}x_{n-1}(t))\,dt.
\end{split}
\end{equation}
Note that this is exactly \eqref{e:det} with $a_{11}=0$. The prey species grows exponentially in the absence of predators. The persistence and extinction of species can still be categorized in the deterministic model \eqref{e:det2} (see \cite[Theorem 5]{GH79}). However, our methods do not work in the stochastic generalization of this setting. We must assume in our proofs that $a_{11}>0$. This assumption is natural from a biological point of view since resources are limited so the prey should not be able to grow without bounds on its own. For results when $a_{11}=0$, in the case when one uses telegraph noise instead of white noise, the reader is referred to \cite{Sato06} where the authors prove a surprising result: switching between two deterministic two-dimensional predator-prey systems does not lead to persistence or extinction. We expect the stochastic version of \eqref{e:det2} with white noise to exhibit similar strange properties.

In ecology there has been an increased interest in the \textit{spatial synchrony} that appears in population dynamics. This refers to the changes in the time-dependent characteristics (i.e. abundances etc) of  populations. One of the mechanisms which creates synchrony is the dependence of the population dynamics on a synchronous random environmental factor such as temperature or rainfall. The synchronizing effect of environmental stochasticity, or the so-called \textit{Moran effect}, has been observed in multiple population models. Usually this effect is the result of random but correlated weather effects acting on populations. For many biotic and abiotic factors, like population density, temperature or growth rate, values at close locations are usually similar. We refer the reader interested in an in-depth analysis of spatial synchrony to \cite{K00, LKB04}.
Most stochastic differential equations models appearing in the population dynamics literature treat only the case when the noise is non-degenerate (although see \cite{R03, DNDY16}). Although this significantly simplifies the technical proofs, from a biological point of view it is not clear that the noise should not be degenerate. For example, if one models a system with multiple populations then all populations can be influenced by the same factors (a disease, changes in temperature and sunlight etc). Environmental factors can intrinsically create spatial correlations and as such it makes sense to study how these degenerate systems compare to the non-degenerate ones. In our setting the noise affecting the different species could be strongly correlated. Actually, in some cases it could be more realistic to have the same one-dimensional Brownian motion $(B_t)_{t\geq 0}$ driving the dynamics of all the interacting species. Therefore, we chose to present a full analysis of the degenerate setting.

\subsubsection {Future work} There are many possible directions for extending our results or adapting them to different settings. One natural generalization would be to work with more general food chains, not necessarily of Lotka-Volterra type, that have been studied in the deterministic setting (see \cite{G80, FS85}) and add environmental fluctuations. We expect that the newly developed methods from \cite{B14} and \cite{HN16} will be key when trying to prove results about persistence and extinction of populations in stochastic environments.

A different problem would be the one where environmental fluctuations are modelled by telegraph noise instead of white noise. This would make our system a piecewise deterministic Markov process (PDMP). These processes have recently been studied in a biological context and have offered new insight regarding the competitive exclusion principle (see \cite{BL16}) and the long term behavior of predator-prey communities (see \cite{M16}).

One can expect that the functional response of some predator species changes according to the seasonally varying prey availability. A first result in this direction has appeared in \cite{TL16}. However, in \cite{TL16} the authors assume that the lengths of the seasons are constant whereas a more intuitive assumption would be that they are random. We expect to explore this direction in future work.

Our model does not account for population structure. Examples of structured populations can be found by looking at a population in which individuals can live in one of $n$ patches (e.g. fish swimming between basins of a lake or butterflies dispersing between meadows).  Dispersion is viewed by many population biologists as an important mechanism for survival. Not only does dispersion allow individuals to escape unfavorable landscapes (due to environmental changes or lack of resources), it also facilitates populations to smooth out local spatio-temporal environmental changes. Lotka-Volterra systems with dispersion in the deterministic setting have been studied in \cite{H78}. It would be interesting generalize this to a stochastic setting.

\bibliographystyle{amsalpha}
\bibliography{LV}

\end{document}